\documentclass[12pt,oneside,reqno]{amsart}

\usepackage[a4paper,left=25mm,top=31mm,right=25mm,bottom=30mm]{geometry}

\usepackage{amssymb,amsfonts}
\usepackage[all,arc]{xy}
\usepackage{enumerate}
\usepackage{siunitx}
\usepackage{enumitem}
\usepackage{mathtools}
\usepackage{mathrsfs}
\usepackage{dsfont}
\usepackage{amsfonts}
\usepackage{amssymb}
\usepackage{graphicx}
\usepackage{ragged2e}
\usepackage{amsthm}
\usepackage{amsmath}
\usepackage[mathscr]{euscript}
 \let\mathscr\relax
\usepackage[scr]{rsfso}
\usepackage{graphicx}
\usepackage{color}
\usepackage{float}
\usepackage{caption}
\usepackage[pdftex,bookmarksnumbered,bookmarksopen,linkcolor=blue,colorlinks=true, citecolor=blue]{hyperref}

\usepackage{mathtools,xparse}

\DeclareMathOperator{\vol}{Vol}

\DeclareMathOperator{\Diam}{Diam}
\DeclareMathOperator{\dist}{Dist}

\newtheorem{thm}{Theorem}[section]

\newtheorem{prop}[thm]{Proposition}
\newtheorem{lem}[thm]{Lemma}

\newtheorem{claim}[thm]{Claim}

\theoremstyle{definition}

\theoremstyle{remark}

\numberwithin{equation}{section}

\author{Andrea Sartori}
\address{ Departement of Mathematics, Tel Aviv University,  IL }
\email{sartori.andrea.math@gmail.com}
\begin{document}
\title{On the universality of the Nazarov-Sodin constant}

\maketitle
\begin{abstract}
We study the number of connected components of non-Gaussian random spherical harmonics on the two dimensional sphere $\mathbb{S}^2$. We prove that the expectation of the nodal domains count is independent of the distribution of the coefficients provided it  has a finite second moment. \end{abstract}
\section{Introduction}

\subsection{Nodal domains of Laplace eigenfunctions}
Let $(M,g)$ be a smooth, compact, connected surface and let $\Delta_g$ be the associated Laplace-Beltrami operators. We are interested in the eigenvalue problem
\begin{align}
	\nonumber	\Delta_g f_{\lambda}+ \lambda f_{\lambda}=0. 
\end{align}
Since $M$ is compact, the spectrum of $-\Delta_g$ is a discrete subset of $\mathbb{R}$ with only accumulation point at $+\infty$. The eigenfunctions $f_{\lambda}$ are smooth and their nodal set, that is their zero set, is a smooth $1d$ sub-manifold outside a finite set of points \cite{C76}. In particular, it is possible to define the nodal domains counting function
$$ \mathcal{N}(f_{\lambda}):= \hspace{3mm} \text{number of connected components of} \hspace{3mm} \{x\in M: f_{\lambda}(x)=0\}.$$

\vspace{2mm}

Courant's nodal domains theorem \cite[Page 21]{Cbook} asserts that there exists some constant $C=C(M)>0$ such that
$$ \mathcal{N}(f_{\lambda})\leq C \lambda.$$
On the other hand, it is not possible to obtain any non-trivial lower  bound for  $\mathcal{N}(\cdot)$. Lewis \cite{L77} showed that there exists a sequence of eigenfunction on the two dimension sphere $\mathbb{S}^2$ with $\mathcal{N}(f_{\lambda})\leq 3$ and $\lambda \rightarrow \infty$. Nevertheless, much of our physical understanding of Laplace eigenfunctions, such as the RWM proposed by Berry \cite{B1,B2,B02} and the percolation prediction of Bogomolny and Schmit \cite{BS2002,BS07}, suggests that, for \textquotedblleft generic\textquotedblright Laplace eigenfunctions, we should expect  
$$ \mathcal{N}(f_{\lambda})\geq c \lambda,$$
for some $c= c(M)$. 

\vspace{2mm}
In order to explore this speculation, Nazarov and Sodin \cite{NS09} studied the number of nodal domains of random Laplace eigenfunctions on $\mathbb{S}^2$. The eigenvalues on $\mathbb{S}^2$ are given by $\lambda=n(n+1)$ for any integer $n>0$ and have multiplicity $2n+1$, thus it is possible to define random Laplace eigenfunctions on $\mathbb{S}^2$ as 
\begin{align}
	\label{function 0}
	f_n(x)= \frac{1}{\sqrt{2n+1}}\sum_{k=-n}^n a_k Y_k(x),
\end{align}
where $a_k$ are i.i.d standard Gaussian random variables and the $Y_k$'s are an orthonormal base of the spherical harmonics of degree $n$, that is Laplace eigenfunction with eigenvalue  $\lambda=n(n+1)$. In this setting,  Nazarov and Sodin \cite{NS09} found that 
$$\mathbb{E}[\mathcal{N}(f)]= c_{NS}n^2(1+o_{n\rightarrow \infty}(1)),$$
where $c_{NS}>0$ is the \textit{Nazarov-Sodin} constant, in agreement with the prediction of  Bogomolny and Schmit.

\vspace{2mm}

 The purpose of this note is to explore  what happens to the expected number of nodal domains when the $a_k$'s in \eqref{function 0} are not Gaussian random variables but i.i.d. with finite second moment. In particular, we are interested in the case when $a_i$ are Bernulli $\pm 1$ random variables. Since (normalized) Laplace eigenfunctions, such as the $Y_k$ in \eqref{function 0}, are defined up to sign, the $\pm 1$ case seems to be a very natural model to study \textquotedblleft generic\textquotedblright Laplace eigenfunctions. 

\subsection{Statement of the main result}
\label{main theorem}
 Given some integer $n>0$, let $\mathcal{H}_n= \mathcal{H}_n(\mathbb{S}^2)\subset L^2(\mathbb{S}^2)$ be the space of spherical harmonics of degree $n$ on $\mathbb{S}^2$, that is the restriction of homogeneous harmonic polynomials of degree $n$ to $\mathbb{S}^2$.  Given an $L^2$ orthonormal basis $\{Y_k\}_{-n\leq k \leq n}$ for $\mathcal{H}_n$, we consider the function
\begin{align}
	\label{function}
	f_{n}(x)= c_n\sum_{k=-n}^n a_{k}Y_{k}(x),
\end{align}
where
\begin{align}
	&\mathbb{E}[a_k]=0 & \mathbb{E}[|a_k|^2]=1 \label{assumptions}
\end{align}
and $c_n= (2n+1)^{-1/2}$ is a normalising constant, which has no impact on the zero set, so that $ \mathbb{E}[|f_{n}|^2]=1.$ With the above notation, we prove the following result: 
\begin{thm}
\label{thm: main}
Let $f_n$ be as in \eqref{function}. Suppose that the $a_k$'s are i.i.d. random variables such that $\mathbb{E}[|a_0|^2]<\infty$, then 
$$\mathbb{E}[\mathcal{N}(f_n)]=c_{NS}n^2(1+o_{n\rightarrow \infty}(1)).$$
\end{thm} 
One crucial difference between the Gaussian and the non-Gaussian case is that the distribution of $f_n$ as in \eqref{function} is \textit{not} base independent (see Claim \ref{claim: not-base indepdent} below). Therefore, most of the Gaussian tools which are pivotal to the \textquotedblleft barrier method\textquotedblright in \cite{NS09} are not available in the non-Gaussian case. Our proof of Theorem \ref{thm: main} takes a different approach which we will now briefly describe. 

\vspace{2mm}
The proof of Theorem \ref{thm: main} comprises essentially of two steps:
\begin{enumerate}
\item The first, and main, step is to show the universality, that is independence from the law of the $a_k$'s in \eqref{function},  for the local nodal domains counting function 
$$\hspace{10mm}\mathcal{N}(F_x)=\text{number of connected components of} \hspace{3mm} \{y\in B(x,O(n^{-1})): f_{\lambda}(y)=0\},$$
where $B(x,O(n^{-1}))$ is a ball centered at $x\in \mathbb{S}^2$ of radius about (we are purposely being slightly vague here) $n^{-1}$. This step requires the most care and its proof is essentially split into three parts:
\begin{enumerate}
\item We show that, for most points $x\in \mathbb{S}^2$, each of the summands in \eqref{function} is $o(n^{1/2})$. Here we use some quite \textquotedblleft elementary\textquotedblright bounds on the $L^4$-norm of spherical harmonics. 
\item We apply a Lindeberg-type CLT to $f_n$ to deduce its (asymptotic) Gaussian behavior outside the aforementioned \textquotedblleft bad\textquotedblright set of $x\in \mathbb{S}^2$.
\item We conclude, using the stability of the nodal set under small perturbations as in \cite{NS}, that 
$$\mathcal{N}(F_x)\overset{d}{\longrightarrow} \mathcal{N}(F_{\mu}) \hspace{5mm} n\rightarrow \infty,$$
where $F_{\mu}$ are Gaussian random waves on the plane and the convergence is in distribution with respect to the product space $\Omega\times \mathbb{S}^2$ (with $\Omega$ being the probability space where the random objects are defined).
\end{enumerate}
\item In the second step, we will use, as in \cite{NS}, the semi-local property of the nodal domains counting functions 
$$\mathcal{N}(f_n)= n (1+o(1))\int_{\mathbb{S}^2} \mathcal{N}(F_x) d\rho(x),$$
where $\rho(\cdot)$ is the uniform measure on $\mathbb{S}^2$, to \textquotedblleft reconstruct\textquotedblright $\mathcal{N}(f_n)$ from our understanding of  $\mathcal{N}(F_x)$ and  conclude the proof.  This step uses some deterministic bounds on the nodal length, i.e. volume of the zero set, of $f_n$ and the Faber-Krahn inequality
\end{enumerate} 

We point out that the universality of the local nodal volume was observed by Kabluchko, Wigman and the author \cite{KSW20} in the study of the nodal volume of non-Gaussian band-limited functions, that is linear combinations of Laplace eigenfunctions. The results in \cite{KSW20} are based on a careful analysis of the spectral projector operator for $\Delta$  and micro-local analysis tools such as Sogge's bound \cite{Sogg88} on the $L^p$-norms of Laplace eigenfunctions. 

On the other hand, the proof of Theorem \ref{thm: main} uses only\textquotedblleft elementary\textquotedblright tools coming from the theory of orthogonal polynomials,  as in \cite{Sbook}, and the  probability language of random functions, as in \cite{BI}. In fact, the exposition in this note is fully self contained a part from a result from \cite{NS09}, which we do not include for the sake of brevity of the article, and the well-known Faber-Krahn inequality (Lemma \ref{lem: Faber-Krahn} below). 

\subsection{Discussion and Related results}
Since the pioneering work of Nazarov and Sodin \cite{NS}, the Nazarov-Sodin constant has been object of much attention. In \cite{NS09}, Nazarov and Sodin extended their work to the number of nodal domains of any ergodic Gaussian stationary field on any (compact and sufficiently smooth) manifold and in \cite{NS20}  they investigated the variance of $\mathcal{N}(f_n)$ in the Gaussian setting. Beliaev, McAuley and Murihead \cite{BMM20,BMM22,BMM22.1} studied , for certain Gaussian random fields $F$,  the number of connected components for the excursion sets $\{F\geq \ell\}$ and level sets $\{F=\ell\}$ for $\ell\neq 0$ and found estimates for the variance and a Central Limit Theorem. As far as nodal domains of random functions are concerned, Theorem \ref{thm: main} seems to be the first result addressing the universality of the nodal domains counting function.  

\vspace{2mm}

The precise value on the Nazarov-Sodin constant is currently unknown. Estimates based on percolation \cite{BS2002,BS07} suggest that 
$$\frac{c_{NS}}{4\pi}= \frac{3\sqrt{3}-5}{\pi}\approx 0.0624...,$$
however this seems slightly inconsistent with numerical simulations \cite{BK13}. Pleijel's nodal domains Theorem \cite{P56}, see also \cite{B15}, asserts that 
$$\frac{c_{NS}}{4\pi}\leq \left(\frac{2}{j_0}\right)^2\approx 0.691,$$
where $j_0$ is the smallest zero of the $0$-th Bessel function. Using some curvature bounds \cite{B02} it is possible to show the sharper bound 
$$\frac{c_{NS}}{4\pi} \leq \frac{1}{\sqrt{2\pi}}\approx 0.0225.$$
As far as lower bounds are concern, the only result we are aware of is by Ingremeau and Rivera \cite{IR19},
$$\frac{c_{NS}}{4\pi} \geq 1.39 \times 10^{-4}.$$

\vspace{2mm}

Finally, we would like to discuss some deterministic lower bounds on the nodal domains counting function. Results in this direction seem few and rare. Ghosh, Reznikov and Sarnak \cite{GRS13,GRS17} gave a (non-trivial) lower bound for the
number of nodal domains for Maass forms on a compact part of the modular surface.   Jung and Zelditch \cite{JZ16} gave a lower bound for a large class of negatively curved surfaces. 

We conclude by discussing the flat two dimensional torus $\mathbb{T}^2=\mathbb{R}^2/ \mathbb{Z}^2$. On $\mathbb{T}^2$, by Fourier expansions, every Laplace eigenfunction with eigenvalue $4\pi^2n^2$ can be expressed as a linear combination of exponentials $\{\exp(2\pi i \xi\cdot x)\}_{|\xi|^2=n}$, where $\xi \in \mathbb{Z}^2$. In particular, it is possible to form linear combinations of eigenfunctions, with the same eigenvalue, as in \eqref{function}. Bourgain \cite{BU}, and the subsequently Buckley and Wigman \cite{BW16}, showed that, for \textquotedblleft most\textquotedblright \textit{deterministic} Laplace eigenfunctions
$$\mathcal{N}(f_n)= c_{NS}n^2(1+_{n\rightarrow \infty}(1)).$$
In particular, Bourgain's result \cite{BU} holds for linear combination of eigenfunctions, as in \eqref{function}, with $Y_k= \exp(2\pi i\xi \cdot)$ and $\pm 1$ coefficients. In other words, on  $\mathbb{T}^2$ when $a_k=\pm 1$,  Theorem \ref{thm: main} holds \textit{deterministically}. 
\subsection{Notation}
\label{sec: notation}
To simplify the exposition we adopt the following standard notation: we write $A\lesssim B$ and $A\gtrsim B$ or $A=O(B)$ to designate the existence of an absolute constant $C>0$ such that $A\leq C B$ and $A\geq C B$. If the said constant $C>0$ depends on some parameter, $\beta$ (say), we write $A\lesssim_{\beta}B$ ecc,  if no parameter is specified in the notation, then the constant is absolute. The letters $C,c$ will be used to designate positive constants which may change from line to line. Finally, we write $o_{\beta\rightarrow \infty}(1)$ for any function that tends to zero as $\beta\rightarrow \infty$.
\section{Tools}
\subsection{Harmonic-analysis tools}
A very useful tool in the study of the nodal set of Laplace eigenfunctions on surfaces is the well-known Faber-Krahn inequality \cite[Chapter 4]{Cbook}, which we state here in a convenient form for our purpose:
\begin{lem}[Faber-Krahn inequality]
	\label{lem: Faber-Krahn}
	Let $f_n$ be as in \eqref{function} and $\Gamma \subset M$ be a nodal domain of $f_n$ with inner radius $r>0$, that is the radius of the largest geodesic ball inscribed in $\Gamma$, then we have 
	$$ r\gtrsim n^{-1}.$$
\end{lem}
We will also need the following consequence of  elliptic regularity \cite[page 336]{Ebook} for harmonic functions. Before state the result, we need to introduce some notation. Recall that $\mathcal{H}_{n}$ is the space of spherical harmonics of degree $n$ and let $B(x,r)$ be the spherical disk centered at $x\in \mathbb{S}^2$ of radius $r>0$. We will need the following result:

\begin{lem}
	\label{lem: bounds} 
	Let $f\in \mathcal{H}_{n}$, $R>0$ be some (large) parameter and $k=0,1,2$,  then 
$$||f||^2_{C^{k}(B(x,R/n))}\lesssim (nR)^{2k+2}\int_{B(x,10R/n)} |f(y)|^2dy,$$
where the constant implied in the $\lesssim$-notation does not depend on $n$, $R$ or $x$.
\end{lem}
For the sake of completeness, we will provide a proof of Lemma \ref{lem: bounds} in Appendix \ref{sec:proof of standard lemmas}. We will also need the following bound on the $L^4$-norm of spherical harmonics, see \cite[Theorem 2]{SW81}:
\begin{lem}
	\label{lem: bound on L^4-norm}
	Let $\{Y_k\}_{k=-n}^n$ be an orthonormal base for $\mathcal{H}_n$, then we have 
	$$\int_{\mathbb{S}^2} |Y_k|^4 d\rho(x)\lesssim  n^{2/3}\log n,$$
	where $d\rho(\cdot)$ is the uniform measure on $\mathbb{S}^2$.
\end{lem}
Although much sharper bounds are known for the $L^4$-norm of Laplace eigenfunction, see for example the fundamental work by Sogge \cite{Sogg88}, Lemma \ref{lem: bound on L^4-norm} will suffice for our purposes and its, rather elementary,  proof will be provided in Appendix \ref{sec:proof of standard lemmas}. 

\subsection{Orthogonal polynomials tools}

 We will also need the following two facts about the two points function of $f_n$ as in \eqref{function}. Again, we need to introduce some notation. Let $P_n(\cdot)$ be the classic Legendre polynomial
\begin{align}
	\label{def: Legendre poly}
	&P_n(x)=\frac{1}{2^n n!}\frac{d^n(x^2-1)^n}{dx^n}  &P_n(1)=1.
\end{align}
Moreover, let $J_0(\cdot)$ be the $0$-th Bessel function
\begin{align}
\label{def: zeroth-bessel function}
J_0(x)=\frac{1}{2\pi}\int_{\mathbb{S}^1} e(x\cdot y)dy,
\end{align}
where $\mathbb{S}^1\subset \mathbb{R}^2$ is the unit circle and $e(\cdot)= \exp(2\pi i \cdot)$. We then have the following, see for example \cite[Page 454]{specialfbook}:
\begin{lem}[Two point function]
	\label{lem: two point function}
	Let $\{Y_k\}_{k=-n}^{n}$ be an orthonormal base for $\mathcal{H}_n$ and $x,y\in \mathbb{S}^2$, then we have 
	$$\frac{1}{2n+1}\sum_{k=-n}^n Y_k(x)Y_k(y)= P_n(\cos \Theta(x,y)), $$
	where $\Theta(x,y)$ is the angle between $x,y$ and $P_n$ is as in \eqref{def: Legendre poly}. 
\end{lem}
The following is \cite[Theorem 8.21.6]{Sbook}:
\begin{lem}
	\label{lem: asymptotic at plank-scale}
Let $P_n$ be as in \eqref{def: Legendre poly} and $J_0$ be as in \eqref{def: zeroth-bessel function}. There exists some absolute constant $C>0$ such that
$$P_n(\cos(\theta))= \frac{\theta}{\sin(\theta)}J_0\left(\left(n+\tfrac{1}{2}\right)\theta\right) +F,$$
where
$$F\lesssim\begin{cases}
\theta^{1/2}n^{-3/2} & C/n\leq \theta \leq \pi/2 \\
\theta^2 &0 \leq \theta \leq C/n
\end{cases}. $$
\end{lem}
\subsection{Probabilistic tools}
We will need a multi-dimensional version of Lindeberg-CLT, see for example \cite[Proposition 6.2]{FK20}: 
\begin{lem}[CLT]
	\label{lem: multidimension clt}
	Let $d>0$ be a positive integer and let $\{V_{n,k}\}_{n,k}$ be a triangular array whose $n$-th row consists of independent $\mathbb{R}^d$-valued random vectors with mean zero. That is, for any fixed $n,k$, $V_{n,k}= (V^{i}_{n,k})_{i=1}^d$ is a $d$-dimension vector. Moreover, we normalize the $V^i_{n,k}$ by writing  $(s^i_{n})^2= \sum_{k} \mathbb{E}[(V^i_{n,k})^2]$ and  $\tilde{V}^i_{n,k}= (s^i_{n})^{-1}V^i_{n,k}$. Assume the following two conditions:
	\begin{enumerate}
		\item Let  $(\Sigma_n)_{ij}= \mathbb{E}[\tilde{V}_{n,i}\tilde{V}_{n,j}]$, be the covariance matrix of the $n$-th row of $\{\tilde{V}_{n,k}\}_{n,k}$, then 
		$$\lim\limits_{n\rightarrow \infty} \Sigma_n= \Sigma_0,$$
		for some positive define $d\times d$-positive matrix.
		\item We have
		$$\max_{i=1,...,d} \frac{1}{(s^i_n)^2}\sum_{k} \mathbb{E}\left[(V^i_{n,k})^2\mathds{1}_{V^i_{n,k}>\varepsilon s^i_n}\right] \rightarrow 0, \hspace{5mm} n\rightarrow \infty,$$
		for any positive $\varepsilon>0$, where $\mathds{1}$ is the indicator function. 
	\end{enumerate}
Then, we have 

$$\sum_k\tilde{V}_{n,k} \overset{d}{\longrightarrow } N(0,\Sigma_0) \hspace{10mm} n\rightarrow \infty,$$
where the convergence is in distribution and the rate of convergence depends on the rates of convergence in (1) and (2) only.  
\end{lem}
We will also need the Continuous Mapping Theorem in the following form, see for example \cite[Theorem 2.1.]{BI}:

\begin{lem}[Continuous mapping]
	\label{lem: continous mapping}
	Let $P_n$, $P$ be probability measures on  a space $(S,\mathcal{S})$, where $\mathcal{S}$ is the Borel $\sigma$-field. Suppose that 
	$$P_n\overset{d}{\longrightarrow} P \hspace{5mm} n\rightarrow \infty,$$
	where the convergence is in distribution, or, in other words, with respect to the weak$^{\star}$ topology. Moreover, let $h:S\rightarrow S$ be a map and let $D_h$ be the set of discontinuity points of $h$. If $P(D_h)=0$, then 
	$$	P_n h^{-1}\overset{d}{\longrightarrow} P h^{-1} \hspace{5mm} n\rightarrow \infty.$$
\end{lem}
Finally, we need the following standard fact about uniform convergence, see \cite[Theorem 3.5]{BI13}
\begin{lem}
	\label{lem: uniform integrability}
Let $X_n$ be a sequence of random variables such that 
	$$X_n\overset{d}{\longrightarrow} X,$$
	where the convergence is in distribution, for some random variable $X$. Suppose that there exists some constants $C,\alpha>0$, independent of $n$ such that $\mathbb{E}[X_n^{1+\alpha}]\leq C$, then 
	$$\mathbb{E}[X_n]\rightarrow \mathbb{E}[X].$$
\end{lem}
\subsection{Gaussian fields tools}
\label{sec: gaussian fields tools}
We will need the main result from \cite{NS}. Before stating it, we introduce some definitions which will be useful throughout the script. Let  $\Omega$ be an abstract probability space, with probability measure $\mathbb{P}(\cdot)$ and expectation $\mathbb{E}[\cdot]$. A (real-valued) Gaussian field $F$ is a map $F: \mathbb{R}^2 \times \Omega\rightarrow \mathbb{R}$ such that all  finite dimensional distributions $(F(x_1, \cdot),...F(x_n,\cdot))$ are multivariate Gaussian vectors and $F(x,\cdot)$ is continuous in $x$. We say that $F$ is \textit{centered} if $\mathbb{E}[F]\equiv0$ and \textit{stationary} if its law is invariant under translations $x\rightarrow x+\tau$ for $\tau \in \mathbb{R}^2$. The \textit{covariance} function of $F$ is 
\begin{align}
	\mathbb{E}[F(x)\cdot F(y)]= \mathbb{E}[F(x-y)\cdot F(0)]. \nonumber
\end{align}
Since the covariance is positive definite, by Bochner's theorem, it is the Fourier transform of some measure $\mu$ on $\mathbb{R}^2$. So we have 
\begin{align}
	\mathbb{E}[F(x)F(y)]= \int_{\mathbb{R}^2} e\left(\langle x-y, s \rangle\right)d\mu(s). \nonumber
\end{align}
The measure $\mu$ is called the \textit{spectral measure} of $F$. Since $F$ is real-valued, $\mu$ is symmetric, that is  $\mu(-A)=\mu(A)$ for any (measurable) subset $A\subset \mathbb{R}^2$. By Kolmogorov's theorem, $\mu$ fully determines $F$. Thus, from now on, we will simply write $F=F_{\mu}$ for the centered, stationary Gaussian field with spectral measure $\mu$. For the rest of the script $\mu$ will \textit{always} denote the uniform measure on $\mathbb{S}^1$ and also let $\mathcal{N}(F_{\mu},R)$ be the number of nodal domain of $F_{\mu}$ fully contained, that is not intersecting the boundary, in $B(R)$, the ball centered at zero of radius $R>0$.  We are finally ready to state the main result from \cite{NS} which we will need:
\begin{thm}[Nazarov-Sodin]
	\label{Thm: NS}
	Let $\mu$, $F_{\mu}$ and $\mathcal{N}(F_{\mu},R)$ be as above. Then, there exists some constant $c_{NS}>0$ such that 
	$$\mathbb{E}[\mathcal{N}(F_{\mu},R)]=c_{NS}\pi R^2 +o(R^2).$$
\end{thm}

We will also need the following standard lemma, see for example \cite[Lemma 6]{NS}:
\begin{lem}[Bulinskaya's lemma]
	\label{lem: Bulinskaya's Lemma}
	Let $F=F_{\mu}$, with $\mu$ be the Lebesgue measure on the unit circle $\mathbb{S}^1$. Then $F\in C^1_*(W)$ almost surely, where is $C^1_*(W)$ as in Lemma \ref{lem: continuity lemma} (below). 
\end{lem}

\subsection{Analysis tools}
We will need the following well-known fact about the nodal length, that is the volume of the zero set, of a function $f_n$ as in \eqref{function}.
\begin{lem}
\label{lem: nodal length}
Let $f_n$ be as in \eqref{function} and let $\mathcal{L}(f_n)=\vol(x\in \mathbb{S}^2: f(x)=0)$, then we have 
$$\mathcal{L}(f_n)\lesssim n.$$
\end{lem}
Although Lemma \ref{lem: nodal length} follows from the much more general result of Donnelly-Fefferman \cite{DF}, we will give an \textquotedblleft elementary\textquotedblright proof in Appendix \ref{sec:proof of standard lemmas}. We will also need the following consequence of Thom's isotopy Theorem, see for example \cite[Theorem 3.1]{EP13} and also \cite[Claim 4.3]{NS09}. We state it in a way convenient for our purposes and will provide a proof, for completeness, in Appendix \ref{sec:proof of standard lemmas}:  
\begin{lem}
	\label{lem: continuity lemma}
	Let $W>0$ be some parameter, $B(W)$ be the ball of radius $W$ centered at the origin and $\partial B(W)$ be its boundary. Let us write 
	
	\begin{align}
		\nonumber
		C^1_{\star}(W):= \{g\in C^1(B(2W))| |g(x)| +|\nabla g(x)|>0 \hspace{3mm} \text{for all} \hspace{3mm} x\in B(W)  \\
		 \text{and} \hspace{3mm} |g| + \left| \nabla g - \frac{x\cdot \nabla g}{|x|^2}x\right|>0 \hspace{3mm} \text{for all} \hspace{3mm} x\in \partial B(W) \}. \nonumber
	\end{align}
	and let $\mathcal{N}(g,W)$   be the number of nodal domain of $g$ fully contained in $B(W)$, that is $\Gamma \cap B(W)=\emptyset$ for all nodal domains of $g$. Then, $N(\cdot,W)$ is a continuous functional on $C^1_{\star}(W)$.
\end{lem}
We comment that the condition $\left| \nabla g - \frac{x\cdot \nabla g}{|x|^2}x\right|>0$ assures that the nodal set does not touch the boundary of $B(W)$ tangentially at one point. In particular, it implies that all the nodal domains that do not intersect $\partial B(W)$ are fully contained in $B(W)$ (and well separated from the boundaries).

\section{Convergence in distribution}
We are finally ready to begin the proof of Theorem \ref{thm: main}. In order to state the main result of this section, we need to introduce some notation. Let $\exp_x:T_x\mathbb{S}^2\cong\mathbb{R}^2 \rightarrow \mathbb{S}^2$ be the exponential map, let use define 
\begin{align}
\label{def: F_x}
	F_{x,R}(y)= F_x(y):= f(\exp_x(Ry/n)),
	\end{align}
where $R>0$ is some parameter and $y\in B(0,1)$, the ball centered at zero of radius $1$. In \eqref{def: F_x}, we tacitly assume that $n$ is much larger than $R$ so that the exponential map is a diffeomorphism. Now, we write $\mathcal{N}(F_x)$ for the number of nodal domains of $F_x$ fully contained in $B(0,1)$. Recalling that $\Omega$ is the abstract probability space where all the objects in the script are defined,  we can think of $\mathcal{N}(F_x)$ as a random variable $\Omega\times \mathbb{S}^2\rightarrow \mathbb{R}$, where $\Omega\times \mathbb{S}^2$ is equipped with the probability measure 
\begin{align}
\label{def: probability measure}
d\sigma = d\mathbb{P} \otimes d\rho,
\end{align}
where $d\rho$ is the uniform probability measure on the unit sphere.

\vspace{3mm}

 The main result of this section is the following:
\begin{prop}
	\label{prop: conv in distribution}
Let $\mu$, $F_{\mu}$, $N(F_{\mu},R)$ be as in section \ref{sec: gaussian fields tools} and $R>0$ some fixed parameter. Then, we have 
$$ \mathcal{N}(F_x) \overset{d}{\longrightarrow} \mathcal{N}(F_{\mu},R) \hspace{5mm} n\rightarrow \infty,$$
where the convergence happens in the cross space $\Omega \times \mathbb{S}^2$ equipped with the measure $d\sigma$ in \eqref{def: probability measure} and the speed of convergence depends on $R$.  
\end{prop}

The proof of Proposition \ref{prop: conv in distribution} relies the observation that $F_x$ is suitable close $F_{\mu}$, as a random function from $(\Omega\times \mathbb{S}^2,d\sigma)$ into $C^1(B(0,1))$, continuously differentiable functions on $B(0,1)$. We rigorously express this claim in the next section after having introduced the relevant notation and background.

\subsection{Convergence of random functions}
\label{sec: convergence of random functions}
 Before stating the main result of this section, we give a brief digression  on the convergence of random functions. Let $C^s(V)$ be the space of $s$-times, $s\ge0$ integer, continuously differentiable functions on $V$, a compact subset of $\mathbb{R}^2$. Since $C^s(V)$ is a separable metric space, Prokhorov's Theorem,  see \cite[Chapters 5 and 6]{BI}, implies that $\mathcal{P}(C^s(V))$, the space of probability measures on $C^s(V)$, is metrizable via the \textit{L\'evy–Prokhorov metric}. This is defined as follows:  for a (measurable) subset $A\subset C^s(V)$, denote by $A_{{+\varepsilon }}$ the  $\varepsilon$-neighborhood of $A$, that is 
$$
A_{{+\varepsilon }}:=\{p\in C^s(V)~|~\exists~ q\in A,\ ||p-q||<\varepsilon \}=\bigcup _{{p\in A}}B(p,\varepsilon),
$$
where $||\cdot||$ is the $C^{s}$-norm and $B(p,\varepsilon)$ is the (open) ball centered at $p$ of radius $\varepsilon>0$. The \textit{L\'evy–Prokhorov metric} $d_P :{\mathcal  {P}(C^s(V))\times\mathcal{P}}(C^s(V))\to [0,+\infty )$ is defined for two probability measures $\mu$  and $\nu$  as:
\begin{align}\nonumber
	d_P (\mu ,\nu ):=\inf \left\{\varepsilon>0: \mu (A)\leq \nu (A_{{+\varepsilon }})+\varepsilon, \ \nu (A)\leq \mu (A_{{+\varepsilon }})+\varepsilon \ \forall~ A\subset C^{s}(V)\right\}. 
\end{align}

\vspace{2mm}

Given an integer $s\geq 1$, $F_x(y)$, as in \eqref{def: F_x}, induces a probability measure on $C^s(B(1))$ via the push-forward measure 
$$ (F_{x})_{\star}\vol(A)=\vol(\{x\in \mathbb{S}^2: F_x(\cdot)\in A\}) ,$$
where $A\subset C^s(B(1))$ is a measurable subset. Similarly, the push-forward of $F_{\mu}$ defines a probability measure on  $C^s(B(1))$ which we denote by  $(F_{\mu})_{\star}\mathbb{P}$. We can now measure the distance between $F_x^B$ and $F_{\mu}$ as the distance between their push-forward measures in  $\mathcal{P}(C^s(B(1))$, the space of probability measures on $C^s(B(1))$,  equipped with the L\'evy–Prokhorov metric. Therefore,  to shorten notation, we will write 
$$ d_P(F_{x},F_{\mu}):= d_P((F_{x})_{\star}d\sigma,(F_{\mu})_{\star}\mathbb{P}).$$

\vspace{2mm}
With the above notation, we have the following result:
\begin{prop}
	\label{prop: convergence of random functions}
	Let $R>0$ be some parameter and $F_x(\cdot)$ be as in \eqref{def: F_x},  $F_{\mu}$ be as in section \ref{sec: gaussian fields tools}. Then, we have 
	$$d_p(F_x,F_{\mu})\rightarrow 0 \hspace{5mm} n\rightarrow \infty,$$
	where $d_p$ is defined with respect to the $C^1$-norm. 
\end{prop}

Since the proof of Proposition \ref{prop: convergence of random functions} is somehow long and technical, we postpone it to section \ref{sec: proof of main prop} below. In the rest of the section we will show how Proposition \ref{prop: conv in distribution} follows from Proposition \ref{prop: convergence of random functions}. 

\subsection{Proof of Proposition \ref{prop: conv in distribution}}
In this section we prove Proposition \ref{prop: conv in distribution} assuming Proposition \ref{prop: convergence of random functions}.
\begin{proof}[Proof of Proposition \ref{prop: conv in distribution} given Proposition \ref{prop: convergence of random functions}]
In order to prove Proposition \ref{prop: conv in distribution} we apply Lemma \ref{lem: continous mapping} to $P_n= (F_{x})_{\star}d\sigma$ and $P= (F_{\mu})_{\star}d\mathbb{P}$ with $h= \mathcal{N}(\cdot)$. By Proposition \ref{prop: convergence of random functions}, we have 
$$P_n \overset{d}{\longrightarrow} P \hspace{5mm} n\rightarrow \infty.$$
Thanks to Lemma \ref{lem: Bulinskaya's Lemma}, applied with $W=10R$ (say), we may assume that $F_{\mu}(R\cdot)\in C^1_{\star}(B(10))$, where $C^1_{\star}$ is as in Lemma \ref{lem: continuity lemma}. Therefore Lemma \ref{lem: continuity lemma} implies that $P(D_h)=0$. Hence, the assumptions of Lemma \ref{lem: continous mapping} are satisfied and Proposition \ref{prop: conv in distribution} follows.   
\end{proof}

\section{Proof of Proposition \ref{prop: convergence of random functions}}
This section is entirely dedicated to the proof of Proposition \ref{prop: convergence of random functions}. The main tool in the proof will be Lemma \ref{lem: multidimension clt}. However, in order to apply the Lindeberg-CLT, we need all the summands in \eqref{function} to have size, before normalization, $o(n^{1/2})$. This is not always the case as there exists spherical harmonics, as $g(\theta,\psi)= \sqrt{2n+1}P_n(\cos(\theta))$, with satisfy $\max_{\mathbb{S}^2}|g|^2\gtrsim n$. In order to circumvent this difficulty, we show that the portion of space where spherical harmonics are large is small. This will be the content of the next section.
\label{sec: proof of main prop}
\subsection{Getting rid of large values}
\label{subsec: bad space}
This section is dedicated to the proof of the following lemma:

\begin{lem}
	\label{lem: getting rid of bad parts of space}
	Let $\{Y_k\}_{k=-n}^n$ be an orthonormal base for $\mathcal{H}_n$, spherical harmonics of degree $n$, and let $R,K>0$ be some (large) parameters. Then there exists a subset $\mathcal{B}\subset \mathbb{S}^2$  of volume at most $O(K^{4}R^{10} n^{-1/3}\log n)$ such that 
	\begin{enumerate}
\item	Uniformly for all $x\not \in \mathcal{B}$, we have 
	$$\max_{-n\leq k\leq n}\sup_{B(x,R/n)}|Y_k(x)| \lesssim K^{-1} n^{1/2} $$
	\item Uniformly for all $x\not \in \mathcal{B}$, we have 
		$$\max_{-n\leq k\leq n}\sup_{B(x,R/n)}|n^{-1}\nabla Y_k(x)| \lesssim K^{-1} n^{1/2}. $$
		Moreover, the constants implied in the $O$ and $\lesssim$-notation are independent of $K,R,x,n$.
\end{enumerate}

\end{lem}
\begin{proof}
	By Lemma \ref{lem: bounds}, we have 
	\begin{align}
\label{eq:3.3} 	\sup_{B(x,R/n)}|Y_k(x)|^2\lesssim (nR)^2 \int_{B(x,10R/n)} |Y_k(y)|^2  dy, \\
	\sup_{B(x,R/n)}|\nabla Y_k(x)|^2\lesssim (nR)^4 \int_{B(x,10R/n)} |Y_k(y)|^2 dy. \nonumber
	\end{align}
Thus, in order to prove Lemma \ref{lem: getting rid of bad parts of space}, it is enough to provide upper an bound on $\int |Y_k|^2$. To this end, we first observe that 
$$\int_{B(x,10R/n)} |Y_k(y)|^2dy \lesssim R^2 n^{-2} \int_{B(0,10)} |Y_{k,x}(y)|^2dy,$$
where $Y_{k,x}(y)=Y_k(\exp_x(Ry/n))$. Moreover,  using Lemma \ref{lem: bound on L^4-norm} and switching the order of integration, we have 
$$\int_{\mathbb{S}^2}\int_{B(0,10)} |Y_{k,x}(y)|^4  dyd\rho(x) \lesssim n^{\tfrac{2}{3}}\log n,$$
where $\rho(\cdot)$ is the uniform measure on $\mathbb{S}^2$. Furthermore, we observe that convexity of the $L^2$-norm (Jensen's inequality) implies
\begin{align}
	\label{eq:3.4}
	\int_{\mathbb{S}^2}\left( \int_{B(x,10)}|Y_{k,x}(y)|^2d\rho(y)\right)^2 d\rho(x)\lesssim \int_{\mathbb{S}^2} |Y_{k,x}(y)|^4d\rho(x) \lesssim  n^{\tfrac{2}{3}}\log n.
\end{align}
 Thus, combining Chebischev's bound, \eqref{eq:3.3} and \eqref{eq:3.4}, we have 
\begin{align}
\label{eq: 3.5}&\rho\left(	\sup_{B(x,R/n)}|Y_k(x)|> K^{-1}n^{1/2}\right) \lesssim 	K^4 n^{-2}\int_{\mathbb{S}^2}\sup_{B(x,R/n)}|Y_k(x)|^4 dx \lesssim K^4 R^8 n^{-{4/3}} \log n \\
&\rho\left(	\sup_{B(x,R/n)}|n^{-1}\nabla Y_k(x)|^2>  K^{-1}n^{1/2} \right) \lesssim K^4 R^{10} n^{-{4/3}} \log n \nonumber
\end{align}
Hence, Lemma \ref{lem: getting rid of bad parts of space} follows by taking the union bound over the $O(n)$ choices of $k$ in \eqref{eq: 3.5}.  
\end{proof}

We are now ready to begin the proof of Proposition \ref{prop: convergence of random functions}. We begin with a standard reduction step, which shows that it is enough to consider finite dimensional distributions to prove convergence of random functions (in our case). 
\subsection{Step I: Reduction}
In order to prove Proposition \ref{prop: convergence of random functions} we will first make a reduction step which relays on the following, well-known, result about tightness of sequences of measures on $\mathcal{P}(C^1(V))$. First, a sequence of probability measures $\{\nu_k\}_{k=0}^{\infty}$ on some topological space $X$ is \textit{tight} if for every $\epsilon>0$, there exists a compact set $K=K(\epsilon)\subset X$ such that
$$\nu_k(X\backslash K)\leq \epsilon,$$
uniformly for all $k\geq 0$. We will need the following lemma, borrowed from \cite[Lemma 1]{Pri93}, see also \cite[Chapter 6 and 7]{BI13}:
\begin{lem}[Tightness]
	\label{lem: tightness}
	Let $V$ be a compact subset of $\mathbb{R}^n$, and $\{\nu_n\}$ a sequence of probability measures on the space $C^1(V)$ of continuously  differentiable functions on $V$. Then $\{\nu_n\}$ is tight if the following conditions hold:
	\begin{enumerate}
		\item For every $|\alpha|\leq 1$, there exists some $y\in V$ such that for every $\varepsilon>0$  there exists $K>0$  with
		\begin{align}
			\nu_n(g\in C^2(V): |D^{\alpha}g(y)|>K)\leq \varepsilon. \nonumber
		\end{align}
		
		\item For every $|\alpha|\leq 1$ and $\varepsilon>0$, we have
		$$\lim_{\delta\rightarrow 0}\limsup_{n\rightarrow \infty}\nu_n\left(g\in C^2(V): \sup_{|y-y'|\leq \delta}|D^{\alpha}g(y)- D^{\alpha}g(y')|> \varepsilon\right)=0.$$
	\end{enumerate}
\end{lem}

As a consequence of Lemma \ref{lem: tightness} we have the following:

\begin{lem}
\label{lem: F_x is tight}
	Let $R>0$ be some parameter and $F_x(\cdot)$ be as in \eqref{def: F_x} and let $\nu_n= (F_x)_{\star}d\sigma$, where $d\sigma$ is the product measure on $\Omega\times \mathbb{S}^2$ and $(F_x)_{\star}$ is the push-forward measure. Then, the sequence $\nu_n$ is tight. 
\end{lem}
\begin{proof}
	In order to check condition $(1)$ in Lemma \ref{lem: tightness}, we observe that, by Lemma \ref{lem: two point function}, we have 
	$$\mathbb{E}[|D^{\alpha}F_x(0)|^2]\lesssim 1.$$
	Thus, by Chebichev's bound, we have 
	$$\sigma \left(|D^{\alpha}F_x(0)|>K \right)\leq \mathbb{P} \left(|D^{\alpha}F_x(0)|>K \right)\lesssim K^{-2},$$
	which implies condition $(1)$ in Lemma \ref{lem: tightness}. 
	
		In order to check condition $(2)$ in Lemma \ref{lem: tightness}, we first observe that 
	\begin{align}
	\label{eq: 4.1}	\sup_{|y-y'|\leq \delta}|D^{\alpha}F_x(y)- D^{\alpha}F_x(y')|\lesssim \sup_{B(0,1)}|\nabla D^{\alpha} F_x| \delta.
	\end{align}
		By Lemma \ref{lem: bounds}, we have 
		$$\sup_{B(0,1)}|\nabla D^{\alpha} F_x|^2\lesssim_{R} \int_{B(0,2)} |F_x(y)|^2dy,$$
		thus, by Lemma \ref{lem: two point function}, we deduce  
		$$\mathbb{E}[\sup_{B(0,1)}|\nabla D^{\alpha} F_x|^2]\lesssim_{R} 1.$$
		Again  by Chebichev's bound and \eqref{eq: 4.1}, we conclude that 
		$$\sigma\left( \sup_{|y-y'|\leq \delta}|D^{\alpha}F_x(y)- D^{\alpha}F_x(y')|\gtrsim K\delta\right)\leq K^{-2},$$
			which, taking $K=\delta^{-1/2}$ say, implies condition $(2)$ in Lemma \ref{lem: tightness}
\end{proof}
In light of Lemma \ref{lem: F_x is tight}, in order to prove Proposition \ref{prop: convergence of random functions} it is enough to prove convergence of the final-dimensional distributions, that is the following lemma:
\begin{lem}
	\label{lem: reduction} 
		Let $R>0$ be some parameter and $F_x(\cdot)$ be as in \eqref{def: F_x}, moreover let $F_{\mu}$ be as in section \ref{sec: gaussian fields tools}. Fix some integer $d>0$ and let $y_1,...,y_d$ be $d$ fixed points in $B(0,1)$, then 
		
		$$(F_x(y_1),...,F_x(y_d))\overset{d}{\longrightarrow}(F_{\mu}(y_1),...,F_{\mu}(y_d)) \hspace{5mm} n\rightarrow \infty,$$
		and for any multi-index $|\alpha|\leq 1$, we also have 
			$$(D^{\alpha}F_x(y_1),...,D^{\alpha}F_x(y_d))\overset{d}{\longrightarrow}(D^{\alpha}F_{\mu}(y_1),...,D^{\alpha}F_{\mu}(y_d)) \hspace{5mm} n\rightarrow \infty.$$
			\end{lem}

\subsection{Step II: Convergence of final dimensional distributions}
In light of the reduction step, we are left with proving Lemma \ref{lem: reduction}, this is the content of this section. 
\begin{proof}[Proof of Lemma \ref{lem: reduction}]
We are first going to focus on the first claim in Lemma \ref{lem: reduction}. By Portmanteau Theorem \cite[Theorem 2.1]{BI} and \cite[Theorem 2.6]{BI}, we have can fix some bounded continuous  function $g: \mathbb{R}^k\rightarrow \mathbb{R}$ (say) and we have to show that 
	\begin{align}
		\label{eq: to prove}
		\int_{\Omega\times \mathbb{S}^2} g((F_x(y_1),...,F_x(y_d)))d\sigma \overset{n\rightarrow \infty}{\longrightarrow} 	\int_{\Omega} g(F_{\mu}(Ry_1),...,F_{\mu}(Ry_d))d\mathbb{P}.
	\end{align}
 In order to prove \eqref{eq: to prove}, let us write 
 $$	\int_{\Omega\times \mathbb{S}^2} g((F_x(y_1),...,F_x(y_d)))d\sigma = \int_{\mathbb{S}^2}d\rho(x)\int_{\Omega}  g((F_x(Ry_1),...,F_x(Ry_d)))d\mathbb{P},$$
 we wish to show that the inner integral converges to the r.h.s. of \eqref{eq: to prove}. This would follow from Lemma \ref{lem: multidimension clt} provided, we can check its assumptions, which we are going to do next.
 
 \vspace{1mm}
 
 Let us first check the convergence of the relative covariance matrix. First, we observe that, for all $y_1,y_2\in B(0,1)$ and $x\in \mathbb{S}^2$ and fixed $R>0$, we have 
 $$  \Theta(\exp_x(Ry_1/n),\exp_x(Ry_1/n))= R\frac{|y_1-y_2|}{n}(1+o_{n\rightarrow \infty}(1)),$$
 where $\Theta(a,b)$ is the angle between $a,b\in \mathbb{S}^2$. Therefore, Lemma \ref{lem: two point function} together with Lemma \ref{lem: asymptotic at plank-scale} and a straightforward differentiation gives that, uniformly for all $x\in \mathbb{S}^2$, we have
 
 \begin{align}
 	\label{eq: conv covariances}&\mathbb{E}\left[D^{\alpha}F_x(y_i)D^{\alpha}F_x(y_j)\right]\longrightarrow 	\mathbb{E}\left[D^{\alpha}F_{\mu}(Ry_i)D^{\alpha}F_{\mu}(Ry_j)\right] &n\rightarrow \infty, 
 \end{align}  
for all $i,j\in \{1,2,...,d\}$ and all multi-indices $|\alpha|\leq 1$. Thus, we have show that the first assumption of Lemma \ref{lem: multidimension clt} holds. 

\vspace{1mm}

Now, in light of the observation at the beginning of section \ref{subsec: bad space}, the second assumption of Lemma \ref{lem: multidimension clt} is \textit{not} satisfied uniformly for all $x\in \mathbb{S}^2$. Thus, we will use Lemma \ref{lem: getting rid of bad parts of space} to get rid of a \textquotedblleft bad\textquotedblright set of $x\in \mathbb{S}^2$, as follows. Let $K=K(n)=\log n$ (say), by Lemma \ref{lem: getting rid of bad parts of space}, applied with such $K$ and $R>0$, there exists some set $\mathcal{B}=\mathcal{B}(n)$ such that 
$$ \rho\left(\mathcal{B}\right) \lesssim_{R} n^{-1/3}(\log n)^{10},$$
and for all $x\not \in \mathcal{B}$ the conclusion of Lemma \ref{lem: getting rid of bad parts of space} holds. Thus, we may re-write the l.h.s. of \eqref{eq: to prove} as
$$\int_{\mathbb{S}^2}d\rho(x)\int_{\Omega}  g((F_x(Ry_1),...,F_x(Ry_d)))d\mathbb{P} =  \int_{\mathbb{S}^2\backslash \mathcal{B}}d\rho(x)\int_{\Omega}  g((F_x(Ry_1),...,F_x(Ry_d)))d\mathbb{P} + o_{g,R}(1),$$
where the error term tends to zero as $n\rightarrow \infty$. Thus, it is enough to prove that 
\begin{align}
\label{eq: enough to prove}
 \int_{\mathbb{S}^2\backslash \mathcal{B}}d\rho(x)\int_{\Omega}  g((F_x(Ry_1),...,F_x(Ry_d)))d\mathbb{P}\rightarrow 	\int_{\Omega} g(F_{\mu}(Ry_1),...,F_{\mu}(Ry_d))d\mathbb{P}.
\end{align}
Hence, it is enough to check the second assumption in Lemma \ref{lem: multidimension clt} holds under the conclusion of Lemma \ref{lem: getting rid of bad parts of space}. This is what we are going to show next. 
\vspace{1mm}
Re-writing the second assumption in Lemma \ref{lem: multidimension clt}, we have to show that
\begin{align}
\label{eq: assumptions lemma}
&\max_{i\in \{0,...,d\}} \frac{1}{2n+1}\sum_{k=-n}^n \mathbb{E}[|a_i Y_k(\exp_x(Ry_i)) \mathds{1}_{|Y_k(\exp_x(Ry_i))|>\varepsilon (2n+1)^{1/2}}]\rightarrow 0 &n\rightarrow \infty,
\end{align}
Since there are $2n+1$ summands in \eqref{eq: assumptions lemma} it enough to prove that, 
$$\max_{i\in \{0,...,d\}} \max_k \mathbb{E}[|a_i Y_k(\exp_x(Ry_i)) \mathds{1}_{|a_iY_k(\exp_x(Ry_i))|^2>\varepsilon (2n+1)^{1/2}}]\rightarrow 0 \hspace{5mm} n\rightarrow \infty,$$
for all $\varepsilon>0$
Since, as discussed above, it is enough to check that the second assumption in Lemma \ref{lem: multidimension clt} holds under the conclusion of Lemma \ref{lem: getting rid of bad parts of space}, we may assume\footnote{Note that here, if necessary, we use the conclusion of Lemma \ref{lem: getting rid of bad parts of space} with $2R$ in place of $R$ so that the image of the exponential map is contained in $B(x,R/n)\subset \mathbb{S}^2$} that 
$$\max_{i\in \{0,...,d\}} |Y_k(\exp_x(Ry_i))|< n^{1/2}(\log n)^{-1}.$$
Therefore, we have 
$$\mathbb{E}[|a_i Y_k(\exp_x(Ry_i)) \mathds{1}_{|a_iY_k(\exp_x(Ry_i))|^2>\varepsilon (2n+1)^{1/2}}]\leq \mathbb{E}[|a_i \mathds{1}_{|a_i|> 10\varepsilon \log n}|^2].$$
Since the $a_i$'s have finite second moment their probability distribution decays at infinity. In order words, we may write 
\begin{align}
	\label{eq: 4.1.3} \mathbb{E}[|a_i \mathds{1}_{|a_i|> 10\varepsilon \log n}|^2] = \int_{\Omega} |a_i|^2\mathds{1}_{|a_i|> 10\varepsilon \log n}d\mathbb{P}(\omega)	=\int_{0}^{\infty} t^2\mathds{1}_{t> 10\varepsilon \log n}d\mathbb{P}(|a_i|>t). 
	 \end{align}
Since the $a_i$'s have finite second moment, by the Dominated Convergence Theorem, we may take the limit $n\rightarrow \infty$ inside the integral in \eqref{eq: 4.1.3}, to see that 
$$\max_i \mathbb{E}[|a_i \mathds{1}_{|a_i|> 10\varepsilon \log n}|^2] \rightarrow 0 \hspace{5mm} n \rightarrow \infty,$$
for all $\varepsilon>0$ (we can actually take $\varepsilon= (\log n)^{-1/2}$, say). This proves \eqref{eq: assumptions lemma}.  
\vspace{1mm}
Hence, in light of \eqref{eq: conv covariances} and \eqref{eq: assumptions lemma}, Lemma \ref{lem: multidimension clt} implies \eqref{eq: enough to prove}, which, in turn, implies \eqref{eq: to prove}, concluding the proof of the first claim in Lemma \ref{lem: reduction}. In light of the second claim in Lemma \ref{lem: getting rid of bad parts of space}, Lemma \ref{lem: two point function} and Lemma \ref{lem: asymptotic at plank-scale}, the second claim of Lemma \ref{lem: reduction} follows form an identical argument and it is therefore omitted. 
\end{proof}
\section{Proof of Theorem \ref{thm: main}}

We are finally ready to conclude the proof of Theorem \ref{thm: main}: 

\begin{proof}[Proof of Theorem \ref{thm: main}]
During the proof, we write $\mathcal{N}(F_x)= \mathcal{N}(F_x, 1)$ and $\mathcal{N}(F_{\mu})= \mathcal{N}(F_{\mu}, R)$. First, observe that,  by Lemma \ref{lem: nodal length}, we have 
$$\mathcal{L}(f_n)\lesssim n$$
Therefore, the number of nodal domains with diameter, that is the largest distance between two points on the said domain, larger than $R/n$ is at most $C_1n^2/R$ for some $C_1>0$. Now suppose that $\Gamma$ is a nodal domain with $\Diam(\Gamma)\leq R/n$, then the volume of $x\in \mathbb{S}^2$ such that $\Gamma \cap B(x,R/n)\neq \emptyset$ but $\Gamma$ is not fully contained in $B(x,R/n)$ is at most $O(\mathcal{L}(\Gamma)\cdot R/n)$, that is the length of $\Gamma$ times the boundary length of $B(x,R/n)$. Therefore, in light of Lemma \ref{lem: nodal length}, we obtain    
 
 $$\left|\sum_{\Diam(\Gamma)\leq R/n}\int \mathds{1}_{\Gamma \subset B(x,R/n)}dx - \pi \frac{R^2}{n^2}\sum_{\Diam(\Gamma)\leq R/n} 1 \right| \lesssim  \frac{R}{n} \sum_{\Gamma} \mathcal{L}(\Gamma)\lesssim R,$$
where $\mathds{1}$ is the indicator function, $\Gamma \subset B(x,R/n)$ means that $\Gamma\cap \partial B(x,R/n)=\emptyset$ and we tacitly assumed that $R>100$ (say). 
Thus, bearing in mind that  the number of nodal domains with diameter larger than $R/n$ is at most $C_1n^2/R$, exchanging the order of summation, we have 
\begin{align}
\label{eq: semi-locality}
\frac{n^2}{\pi R^2}	\int \mathcal{N}(F_x) d\rho(x) &= \sum_{ \Diam(\Gamma)\leq R/n} \frac{n^2}{\pi R^2}\int \mathds{1}_{\Gamma \in B(x,R/n)}dx \nonumber \\
 &= \sum_{\Gamma} 1 +  O\left(\frac{n^2}{R} + R\right) = \mathcal{N}(f_n) +  O\left(\frac{n^2}{R}\right),
\end{align}
as $R$ is assumed to be much smaller than $n$. 
\vspace{1mm}

Now, by Lemma \ref{lem: Faber-Krahn}, we have 
$$\int_{\mathbb{S}^2\times \Omega} \mathcal{N}(F_x)^2 d\sigma \lesssim R^2.$$
Therefore, in light of Proposition \ref{prop: conv in distribution}, we may apply Lemma \ref{lem: uniform integrability} with $X_n=  \mathcal{N}(F_x)$, $X= \mathcal{N}(F_{\mu}(R\cdot))$, where $\mu$ is the uniform measure on $\mathbb{S}^1$, to see that 
\begin{align}
	\label{eq: almost finished}
	\int_{\mathbb{S}^2\times \Omega} \mathcal{N}(F_x) d\sigma = \mathbb{E}[\mathcal{N}(F_{\mu})](1+o_{n\rightarrow \infty}(1)).
\end{align}
Combining \eqref{eq: semi-locality} and \eqref{eq: almost finished}, we  deduce that 
$$ \mathbb{E}[\mathcal{N}(f_n)]= \frac{n^2}{\pi R^2}  \mathbb{E}[\mathcal{N}(F_{\mu})](1+o_{n\rightarrow \infty}(1)) + O\left(\frac{n^2}{ R^2}\right).$$
Hence, thanks to Theorem \ref{Thm: NS}, we have 
$$  \mathbb{E}[\mathcal{N}(f_n)]= c_{NS}n^2 (1+o_{n\rightarrow \infty}(1))(1+o_{R\rightarrow \infty}(1)) + O\left(\frac{n^2}{R^2}\right),$$  
 and Theorem \ref{thm: main} follows by taking $R\rightarrow\infty$ sufficiently slowly compared to $n$. 
\end{proof}

\appendix

\section{Proof of auxiliary lemmas}
\label{sec:proof of standard lemmas}
We will now prove Lemma \ref{lem: bounds}.
\begin{proof}[Proof of Lemma \ref{lem: bounds}]
	Let us consider the \textquotedblleft harmonic lift\textquotedblright $h(x,t)= \exp(-n(n+1)^{1/2}t)f_n(x): \mathbb{S}^2\times [-2,2]\rightarrow \mathbb{R}$. Note that,  $\Delta h= 0$ on (say) any ball contained in the product space $\mathbb{S}^2\times [-2,2]$ . Thus, by the mean value property of harmonic functions, for all $y\in \mathbb{S}^2$, we have 
	  $$f_n(y)=h(y,0)\lesssim \vol (\tilde{B}((y,0), n^{-1}))\int_{\tilde{B}((y,0),n^{-1})} h,$$
where $\tilde{B}(\cdot)$ denotes a ball in the product space.  Using Cauchy-Schwartz and integrating over the auxiliary variable ($t$), we obtain 
 \begin{align}
 	\label{eq: A 1.1}
 	|f_n(y)| \lesssim n \left( \int_{B(y, n^{-1})} |f_n|^2 \right)^{1/2}.
 \end{align}
Applying \eqref{eq: A 1.1} (squared) to every point in $y\in B(x,R/n)$, we conclude 
$$ \sup_{B(x,R/n)} |f_n|^2\lesssim (nR)^2 \int_{B(x,10R/n)} |f_n|^2.$$

\vspace{1mm}
We now consider bounds on the the derivatives of $f_n$. Since $h$ (as above) is harmonic, we also have $D^{\alpha}h$ is harmonic for all multi-indices $|\alpha|\leq 1$ say. Therefore, by the mean-value property of harmonic functions and the Divergence Theorem, for all points $w\in \mathbb{S}^2 \times[-1,1]$ and a ball $\tilde{B}(W)= \tilde{B}(w,W)$ with $W>0$, writing $D^{\alpha}= D$, we find 
$$D h(w)= \vol (\tilde{B}(W))^{-1} \int_{\tilde{B}(W)} D h = \vol (\tilde{B}(W))^{-1} \int_{\partial \tilde{B}(W)} h\cdot n,$$
where $n$ is the outward pointing unit-norm vector. In particular, we have 
$$D h(w) \lesssim W^{-1}\sup_{\tilde{B}(2W)}|h|.$$
Taking $w= (y,0)$, $W= 1/(2n)$ and using \eqref{eq: A 1.1}, we obtain
$$|D f_n(y)|^2\lesssim n^4 \sup_{B(y,1/n)}|f_n|^2.$$
Another covering argument as above concludes the proof of Lemma \ref{lem: bounds} for the $C^1$-norm. To see Lemma \ref{lem: bounds}  for the $C^2$-norm we simply use the  mean-value property and  the Divergence Theorem with $D^{\beta}(D^{\alpha}h)$, for $|\alpha|\leq 1$ and $|\beta|\leq 1$, to obtain 
$$D^{\beta}(D^{\alpha}h)(y,1)\lesssim n^2 \sup_{B((y,1),n^{-1})} |D^{\alpha}h|\leq n^6 \sup_{B(y,1/n)}|f_n|^2,$$
and repeat the covering argument.
\end{proof}

We will now prove Lemma \ref{lem: bound on L^4-norm} following \cite[Theorem 2]{SW81}, we claim no originality. 
\begin{proof}[Proof of Lemma \ref{lem: bound on L^4-norm}]
Let us suppose that $u$ is a function which maximizes $\int_{\mathbb{S}^2} u^4 $ among all functions $u\in \mathcal{H}_n$ with $||u||_{L^2}=1$ (note that $u$ exists since $\mathcal{H}_n$ is a finite dimensional vector space). Now, 	let us consider the integral kernel of the spectral projector operator $\pi_n:L^2(\mathbb{S}^2)\rightarrow \mathcal{H}^n$ which, in light of Lemma \ref{lem: two point function}, is given by 
	$$\sum_{k=-n}^n Y_k(x)Y_k(y)= (2n+1) P_n(\langle x,y\rangle):= \varphi_n(\langle x,y\rangle),$$
	where $\langle \cdot,\cdot\rangle$ is the standard inner-product on $\mathbb{S}^2$ so that $\cos(\Theta(x,y))= \langle \cdot,\cdot\rangle$. Then, by our choice of $u$, we claim that
	\begin{align}
		\label{eq: A1}
\pi_n(u^3)(y):=	\varphi_n \star u^3 (y)=	\int_{\mathbb{S}^2} \varphi_n(\langle x,y\rangle) u(y)^3d\rho(y)= c u(x),
	\end{align}
	for some constant $c>0$. Indeed, for $f\in L^2(\mathbb{S}^2)$, we have 
	\begin{align}
		\int_{\mathbb{S}^2} f\cdot u= \int_{\mathbb{S}^2} \pi_n(f)\cdot u + \int_{\mathbb{S}^2} (f-\pi_n(f))\cdot u= \int_{\mathbb{S}^2} \pi_n(f)\cdot u  , \label{eq: A2}
		\end{align}
as $u\in \mathcal{H}_n$ and $ (f-\pi_n(f))$ is orthogonal to it.  Thus, since the integral on the  r.h.s. of \eqref{eq: A2} is maximized for $\pi_n(f)= cu$ and $u$ by definition maximizes the integral on the l.h.s. of \eqref{eq: A2}, we conclude \eqref{eq: A1}. 

\vspace{1mm} 
In order to find $c>0$ in \eqref{eq: A1}, we multiply both sides by $u(x)$ and integrate with respect to $x\in \mathbb{S}^2$. Thus, bearing in mind that $||u||_{L^2}=1$, we find
$$c= ||u||^4_{L^4}.$$
Therefore, \eqref{eq: A1} and H\"{o}lder inequality (for the $L^1$-norm) implies 
\begin{align}
	\nonumber ||u||^4_{L^4}||u||_{L^{\infty}}= 	||\varphi_n \star u^3||_{L^{\infty}} \leq 	||\varphi_n||_{L^4}||u^3||_{L^{3/4}}= 	|| 	\varphi_n||_{L^4}||u||^3_{L^4}.
	\end{align}
In particular,  we obtain the pair of bounds 
\begin{align}
\label{eq: A3}	& ||u|_{L^4}||u||_{L^{\infty}}\leq 	|| \varphi_n||_{L^4} &||u||^4_{L^4}\leq ||u||^2_{L^{\infty}}.
\end{align}
Hence, if $||u||_{L^{\infty}}\leq n^{1/3}$, then \eqref{eq: A3} implies the bound $||u||^4_{L^4}\leq n^{2/3}$. If  $||u||_{L^{\infty}}> n^{1/3}$, since a straightforward computation using Lemma \ref{lem: asymptotic at plank-scale}  implies  $|| \varphi_n||_{L^4}\lesssim n^{1/2}\log n$, \eqref{eq: A3} implies  $||u||^4_{L^4}\lesssim n^{2/3}\log n$, as required. 
\end{proof}
We will now prove Lemma \ref{lem: nodal length}:
\begin{proof}[Proof of Lemma \ref{lem: nodal length}]
	First, since spherical harmonics are restrictions of homogeneous polynomials to the sphere, by passing to polar coordinates, we may identify $f_n$ with a bi-variate trigonometric polynomial $g_n$ (say) so that  
	$$\vol (x\in \mathbb{S}^2: f_n(x)=0)\asymp \vol (x\in [-1/2,1/2]^2: g_n(x)=0),$$
	Therefore, it is enough to prove 
	\begin{align}
		\label{eq: prove}
	 \vol (x\in [-1/2,1/2]^2: g_n(x)=0)\lesssim n.
	\end{align}
Now we claim that there exists either an horizontal line $\ell_h$ (say) or a vertical line $\ell_v$ (say) such that 
\begin{align}
\label{eq: Crofton like}	 \vol (x\in [-1/2,1/2]^2: g_n(x)=0) \lesssim |\{x\in \mathcal{L}_h: g(x)=0 \}| +  |\{x\in \mathcal{L}_v: g(x)=0 \}|
	 \end{align}
 	Since the zero set of $g_n$ is an union of smooth curves, it is, in particular, rectifiable (approximable by line segments). So it is enough to show that for any line segment $\mathcal{C}$ with length $\ell$ claim \ref{eq: Crofton like} holds. Indeed, let us write $N_1(x_1)$ for the number of intersections of $\mathcal{C}$ with the horizontal line going through $x_1\in [-1/2,1/2]$ (the lower side of the square $[-1/2,1/2]^2)$ and, similarly,  $N_1(x_2)$ for the number of intersections of $\mathcal{C}$ with the vertical line going through $x_2\in [-1/2,1/2]$ (the left side of the square $[-1/2,1/2]^2)$. Then, we have 
 	$$\int N(x_1)dx_1 + \int N(x_2)dx_2 \geq P_1(\mathcal{C}) + P_2(\mathcal{C})\geq \frac{\ell}{10},$$
 	where $P_1,P_2$ are the length of the projection of $\mathcal{C}$ on the $X$ and $Y$-axis respectively, and Claim \ref{eq: Crofton like} follows. 
 	  Since $g_n$ is a bi-variate polynomial of degree at most $n$, its restriction to any vertical or horizontal line is an uni-variate polynomial of degree $n$, thus, by \eqref{eq: Crofton like}, we have 
	  $$\vol (x\in [-1/2,1/2]^2: g_n(x)=0) \lesssim n,$$
	as required. 
\end{proof}
We are now going to prove Lemma \ref{lem: continuity lemma}, the proof follows \cite[Claim 4.2]{NS}, we claim no originality. 
\begin{proof}[Proof of Lemma \ref{lem: continuity lemma}]
	Let $\{g_{k}\}_{k=1}^{\infty}\in C^1_{\star}$ be a sequence of functions converging to some $g\in C^1_{\star}(W)$ with respect to the $C^1$-topology. Since $g\in C^1_{\star}$, $g$ has finitely many nodal domains in $B(2W)$. In particular, it has finitely many nodal domains which do not touch $\partial B(W)$ and, in light of the definition of $C^1_{\star}$, there exists some $a>0$ such that 
	$$\min \dist(\Gamma, \partial B(W))>a,$$
	where the minimum  is take over all nodal domains $\Gamma$ of $g$ contained in $B(W)$.  Moreover, there exists some $b$ such that 
	$$|g|+|\nabla g|>b.$$
	
	\vspace{1mm}
	
	Now, we claim that each connected component $\Gamma(t)$ of the level set $|g|\leq t$ contains precisely one nodal domain $\Gamma$ and the $\Gamma(t)$'s are disjoint, provided we choose $t=t(a,b)$ sufficiently small.	Let's start by showing that the $\Gamma(t)$'s are disjoint. If they were not, they would meet at a point $x$ (say) where $|\nabla g(x)|=0$, since $x\in B(2W)$, this implies that $|g(x)|>b$ and thus it does not belong to $\Gamma(b/2)$.
	
	Let us now snow that each connected component $\Gamma(t)$ of the level set $|g|\leq t$ contains precisely one nodal domain $\Gamma$,  provided we choose $t=t(a,b)$ sufficiently small.  Indeed, since $g$ is continuous, taking $t$ sufficiently small depending on $a$, we may assure that 
	$$	\min \dist(\Gamma(t), \partial B(W))>a/2.$$
So all $\Gamma(t)$ as sufficiently well separated from the boundaries.  Now, let $x\in \partial\Gamma(t)$, then $|g(x)|=t$ and, taking $t<b/2$, we see that $|\nabla g(x)|>b/2$, thus moving in the direction of $\nabla g(x)$ if $g(x)=-t$ and in the direction $\nabla g(x)$ if $g(x)=t$, we will find a point (within $\Gamma(t)$ for appropriately small $t$ depending on $b$) such that $g(x)=0$.  Since $|g|+|\nabla g|>b$, there are exists some $c=c(b)>0$ such that 
	$$\min \dist (\Gamma_i,\Gamma_j)>c.$$
	By continuity of $g$, we may choose $t$ sufficiently small such that $\Gamma(t)\subset \Gamma_{+(c/2)}$, the $(c/2)$-neighborhood of $\Gamma$ ( and all this neighborhood are well within $B(W)$, provided with also choose $c$ small compared to $a$). Thus, $\Gamma(t)$ contains precisely one nodal domain. 
	
	\vspace{1mm}
	To conclude the proof we take $k$ sufficiently large so that each nodal domains of $g_k$ is contained in $\Gamma(t)$, for some $\Gamma(t)$,  which implies, as $\mathcal{N}(\cdot)$ is integer values, that 
	$$\mathcal{N}(g_k,W)=\mathcal{N}(g,W),$$
	as required. 
\end{proof}
\section{Not base-independent}
We prove that the distribution of $f_n$ as a random function on $C^0(\mathbb{S}^2)$ (say) is not base independent:
\begin{claim}
	\label{claim: not-base indepdent}
  Given and orthogonal $\{Y_k\}$ for $\mathcal{H}_n$ let $v_n= (f_n)_{\star}\mathbb{P}$ be as in section \ref{sec: convergence of random functions} with 
	$$f_n(x)= c_n \sum_{k=-n}^n a_k Y_k,$$
	where the  $a_k$'s are i.i.d. Bernulli $\pm 1$. There exists two (different) orthonormal basis $\{Y_k\}_{k=-n}^n$ and $\{\tilde{Y_k}\}_{k=-n}^n$ such that their associated pushfoward measures, $v_n$ and $\tilde{v_n}$ (say), have different distributions.  
\end{claim}
\begin{proof}
	Suppose, by contradiction that $v_n=\tilde{v_n}$, in the sense of distributions. Then, their finite dimensional distributions also agree, in particular, taking $x=(0,0)$ to be the north pole, in polar coordinates on the sphere, we should have 
	\begin{align}
	\label{eq: to disprove}
	\mathbb{P}(f_n(x)\leq t)= \mathbb{P}(\tilde{f_n}(x)\leq t),
	\end{align}
for all $t\in \mathbb{R}$. Let us now take $Y_k(\theta,\psi)= \exp(ik\psi)P_{n}^k(\cos \theta)$, where $(\theta,\psi)$ are polar coordinates on $\mathbb{S}^2$ and
$$P_n^k(x)=\left(\frac{n(n-k)!}{(n+k)!}\right)^{1/2}\frac{(-1)^k}{2^{k}k!}(1-x^2)^{k/2}\frac{d^{n+k}}{dx^{n+k}}(x^2-1)^{n},$$
are the (normalized) associated Legendre polynomials.
Moreover,  let us take 
$$   \tilde{Y_k}= Y_k \hspace{3mm} k\neq {0,1}, \hspace{5mm}  \tilde{Y_0}=\frac{1}{\sqrt{2}}(Y_0+ Y_1), \hspace{5mm} \tilde{Y_1}=\frac{1}{\sqrt{2}}(Y_0- Y_1).$$
By definition, at the north pole $Y_k(x)=0$ for all $k\neq 0$ and $Y_0(x)=\sqrt{2n+1}$, thus, bearing in mind the normalization constant in the definition of $f_n$, we obtain 
$$	\mathbb{P}(f_n(x)\leq t)=\mathbb{P}(a_0\leq t),$$
and 
$$\mathbb{P}(\tilde{f_n}(x)\leq t) = \mathbb{P}(2^{-1/2}(a_0+ a_1)\leq t),$$
which contradicts \eqref{eq: to disprove}.
\end{proof}
\section*{Acknowledgements}
We would like to thank Mikhail Sodin for suggesting the question explored in this script and for many useful discussions. This work was supported by the ISF Grant 1903/18 and the BSF Start up Grant no. 20183

\bibliographystyle{siam}
\bibliography{UnvNS}

\begin{thebibliography}{10}

\bibitem{specialfbook}
{\sc G.~E. Andrews, R.~Askey, and R.~Roy}, {\em Special functions}, vol.~71 of
  Encyclopedia of Mathematics and its Applications, Cambridge University Press,
  Cambridge, 1999.

\bibitem{BK13}
{\sc D.~Beliaev and Z.~Kereta}, {\em On the {B}ogomolny-{S}chmit conjecture},
  J. Phys. A, 46 (2013), pp.~455003, 5.

\bibitem{BMM20}
{\sc D.~Beliaev, M.~McAuley, and S.~Muirhead}, {\em On the number of excursion
  sets of planar {G}aussian fields}, Probab. Theory Related Fields, 178 (2020),
  pp.~655--698.

\bibitem{BMM22.1}
\leavevmode\vrule height 2pt depth -1.6pt width 23pt, {\em A central limit
  theorem for the number of excursion set components of gaussian fields}, Arxiv
  preprint: 2205.09085,  (2022).

\bibitem{BMM22}
\leavevmode\vrule height 2pt depth -1.6pt width 23pt, {\em Fluctuations of the
  number of excursion sets of planar {G}aussian fields}, Probab. Math. Phys., 3
  (2022), pp.~105--144.

\bibitem{B1}
{\sc M.~V. Berry}, {\em Regular and irregular semiclassical wavefunctions},
  Journal of Physics A: Mathematical and General, 10 (1977), p.~2083.

\bibitem{B2}
\leavevmode\vrule height 2pt depth -1.6pt width 23pt, {\em Semiclassical
  mechanics of regular and irregular motion}, Les Houches lecture series, 36
  (1983), pp.~171--271.

\bibitem{B02}
{\sc M.~V. Berry}, {\em Statistics of nodal lines and points in chaotic quantum
  billiards: perimeter corrections, fluctuations, curvature}, J. Phys. A, 35
  (2002), pp.~3025--3038.

\bibitem{BI}
{\sc P.~Billingsley}, {\em Convergence of probability measures}, John Wiley \&
  Sons, 2013.

\bibitem{BI13}
\leavevmode\vrule height 2pt depth -1.6pt width 23pt, {\em Convergence of
  probability measures}, John Wiley \& Sons, 2013.

\bibitem{BS2002}
{\sc E.~Bogomolny and C.~Schmit}, {\em Percolation model for nodal domains of
  chaotic wave functions}, Phys. Rev. Lett., 88 (2002), p.~114102.

\bibitem{BS07}
{\sc E.~Bogomolny and C.~Schmit}, {\em Random wavefunctions and percolation},
  Journal of Physics A: Mathematical and Theoretical, 40 (2007),
  pp.~14033--14043.

\bibitem{BU}
{\sc J.~Bourgain}, {\em On toral eigenfunctions and the random wave model},
  Israel Journal of Mathematics, 201 (2014), pp.~611--630.

\bibitem{B15}
\leavevmode\vrule height 2pt depth -1.6pt width 23pt, {\em On {P}leijel's nodal
  domain theorem}, Int. Math. Res. Not. IMRN,  (2015), pp.~1601--1612.

\bibitem{BW16}
{\sc J.~Buckley and I.~Wigman}, {\em On the number of nodal domains of toral
  eigenfunctions}, in Annales Henri Poincar{\'e}, vol.~17, Springer, 2016,
  pp.~3027--3062.

\bibitem{Cbook}
{\sc I.~Chavel}, {\em Eigenvalues in {R}iemannian geometry}, vol.~115 of Pure
  and Applied Mathematics, Academic Press, Inc., Orlando, FL, 1984.
\newblock Including a chapter by Burton Randol, With an appendix by Jozef
  Dodziuk.

\bibitem{C76}
{\sc S.~Y. Cheng}, {\em Eigenfunctions and nodal sets}, Comment. Math. Helv.,
  51 (1976), pp.~43--55.

\bibitem{DF}
{\sc H.~Donnelly and C.~Fefferman}, {\em Nodal sets of eigenfunctions on
  reimannian manifolds}, Inventiones mathematicae, 93 (1988), pp.~161--183.

\bibitem{EP13}
{\sc A.~Enciso and D.~Peralta-Salas}, {\em Submanifolds that are level sets of
  solutions to a second-order elliptic {PDE}}, Adv. Math., 249 (2013),
  pp.~204--249.

\bibitem{Ebook}
{\sc L.~C. Evans}, {\em Partial differential equations}, vol.~19 of Graduate
  Studies in Mathematics, American Mathematical Society, Providence, RI, 1998.

\bibitem{FK20}
{\sc H.~Flasche and Z.~Kabluchko}, {\em Expected number of real zeroes of
  random {T}aylor series}, Commun. Contemp. Math., 22 (2020), pp.~1950059, 38.

\bibitem{GRS13}
{\sc A.~Ghosh, A.~Reznikov, and P.~Sarnak}, {\em Nodal domains of {M}aass forms
  {I}}, Geom. Funct. Anal., 23 (2013), pp.~1515--1568.

\bibitem{GRS17}
\leavevmode\vrule height 2pt depth -1.6pt width 23pt, {\em Nodal domains of
  {M}aass forms, {II}}, Amer. J. Math., 139 (2017), pp.~1395--1447.

\bibitem{IR19}
{\sc M.~Ingremeau and A.~Rivera}, {\em A lower bound for the
  {B}ogomolny-{S}chmit constant for random monochromatic plane waves}, Math.
  Res. Lett., 26 (2019), pp.~1179--1186.

\bibitem{JZ16}
{\sc J.~Jung and S.~Zelditch}, {\em Number of nodal domains and singular points
  of eigenfunctions of negatively curved surfaces with an isometric
  involution}, J. Differential Geom., 102 (2016), pp.~37--66.

\bibitem{KSW20}
{\sc Z.~Kabluchko, A.~Sartori, and I.~Wigman}, {\em Expected nodal volume for
  non-gaussian random band-limited functions}, Arxiv preprint:
  https://arxiv.org/abs/2102.11689,  (2020).

\bibitem{L77}
{\sc H.~Lewy}, {\em On the minimum number of domains in which the nodal lines
  of spherical harmonics divide the sphere}, Comm. Partial Differential
  Equations, 2 (1977), pp.~1233--1244.

\bibitem{NS09}
{\sc F.~Nazarov and M.~Sodin}, {\em On the number of nodal domains of random
  spherical harmonics}, Amer. J. Math., 131 (2009), pp.~1337--1357.

\bibitem{NS}
{\sc F.~Nazarov and M.~Sodin}, {\em Asymptotic laws for the spatial
  distribution and the number of connected components of zero sets of gaussian
  random functions}, J. Math. Phys. Anal. Geom., 12 (2016), pp.~205--278.

\bibitem{NS20}
{\sc F.~Nazarov and M.~Sodin}, {\em Fluctuations in the number of nodal
  domains}, J. Math. Phys., 61 (2020), pp.~123302, 39.

\bibitem{P56}
{\sc A.~k. Pleijel}, {\em Remarks on {C}ourant's nodal line theorem}, Comm.
  Pure Appl. Math., 9 (1956), pp.~543--550.

\bibitem{Pri93}
{\sc S.~M. Prigarin}, {\em Weak convergence of probability measures in the
  spaces of continuously differentiable functions}, Sibirskii Matematicheskii
  Zhurnal, 34 (1993), pp.~140--144.

\bibitem{Sogg88}
{\sc C.~D. Sogge}, {\em Concerning the {$L^p$} norm of spectral clusters for
  second-order elliptic operators on compact manifolds}, J. Funct. Anal., 77
  (1988), pp.~123--138.

\bibitem{SW81}
{\sc R.~J. Stanton and A.~Weinstein}, {\em On the {$L^{4}$} norm of spherical
  harmonics}, Math. Proc. Cambridge Philos. Soc., 89 (1981), pp.~343--358.

\bibitem{Sbook}
{\sc G.~Szeg\H{o}}, {\em Orthogonal polynomials}, American Mathematical Society
  Colloquium Publications, Vol. XXIII, American Mathematical Society,
  Providence, R.I., fourth~ed., 1975.

\end{thebibliography}
	\end{document}